\documentclass[11pt]{amsart}

\usepackage{amssymb, amsthm, amsmath}
\usepackage{color}
\newtheorem{theorem}{Theorem}[section]

\newtheorem{lemma}[theorem]{Lemma}

\newtheorem{prop}[theorem]{Proposition}

\newtheorem*{theorem*}{Theorem}{\bf}{\it}
\newtheorem*{proposition*}{Proposition}{\bf}{\it}
\newtheorem*{observation*}{Observation}{\bf}{\it}
\newtheorem*{lemmaA*}{Lemma A}{\bf}{\it}
\newtheorem*{lemmaB*}{Lemma B}{\bf}{\it}

\theoremstyle{definition}

\theoremstyle{remark}
\newtheorem{remark}[theorem]{Remark}

\newcommand{\h}{\mathcal H}

\newcommand{\R}{\mathbb R}

\newcommand{\K}{\mathcal{K}}
\newcommand{\dv}{{\rm{div}}}
\newcommand{\dist}{{\rm{dist}}}
\newcommand{\diam}{{\rm{diam}}}

\def\XXint#1#2#3{{\setbox0=\hbox{$#1{#2#3}{\int}$ }
\vcenter{\hbox{$#2#3$ }}\kern-.6\wd0}}

\begin{document}
\title[Propagation of smallness]{Quantitative propagation of smallness for solutions of  elliptic equations.}
\author[A. Logunov]{Alexander Logunov}
\address{A.L.: School of Mathematical Sciences, Tel Aviv University, Tel Aviv, 69978, Israel;}
\address{Chebyshev Laboratory, St. Petersburg State University, 14th Line V.O., 29B, Saint Petersburg, 199178, Russia}
\address{Institute for Advanced Study, Princeton, NJ, 08540, USA}
\email{log239@yandex.ru}

\author[E. Malinnikova]{Eugenia Malinnikova}
\address{E.M.: Department of Mathematics, Purdue University, 
150 N. University Street, West Lafayette, IN 47907-2067}
\address{Department of Mathematical Sciences,
Norwegian University of Science and Technology
7491, Trondheim, Norway}
\email{eugenia.malinnikova@ntnu.no}
\begin{abstract}  Let $u$ be a solution to an elliptic equation $\textup{div}(A\nabla u)=0$ with Lipschitz coefficients in $\mathbb{R}^n$.
 Assume $|u|$ is bounded by $1$ in the ball $B=\{|x|\leq 1\}$. We show that if $|u| < \varepsilon$ on a set  $ E \subset \frac{1}{2} B$
 with positive $n$-dimensional Hausdorf measure, then
$$|u|\leq C\varepsilon^\gamma \textup{ on } \frac{1}{2}B,$$
 where $C>0, \gamma \in (0,1)$ do not depend on $u$ and depend only on $A$ and the measure of $E$.
 We specify the dependence on the measure of $E$ in the form of the Remez type inequality. Similar estimate holds for sets $E$ with Hausdorff dimension bigger than $n-1$.

 For the gradients of the solutions we show that a similar propagation of smallness holds for  sets of
 Hausdorff dimension bigger than $n-1-c$, where $c>0$ is a small numerical constant depending on the dimension only.

\end{abstract}
\thanks{The author was supported in part by ERC Advanced Grant 692616, ISF
Grants 1380/13, 382/15 and by a Schmidt Fellowship at the Institute for Advanced
Study. \\
E. M. was supported by Project 213638 of the Research Council of Norway.}
\maketitle
\section{Introduction}

  This paper contains several quantitative results on  propagation of smallness for solutions of elliptic PDE.
The results concern the logarithms of the magnitudes of the solutions and their gradients. The techniques used in this paper were recently applied to estimates of zero sets of Laplace eigenfunctions \cite{LM1},\cite{L2},\cite{L3}.

 The inspiration comes from  the following useful fact from complex analysis:
{\it if $f$ is a holomorphic function on $\mathbb{C}$, then $\log|f|$ is subharmonic.}
 For this simple and powerful fact there are no known direct analogs for real valued solutions of elliptic PDE on $\mathbb{R}^n$, except for
  the gradients of harmonic functions  on $\mathbb{R}^2$, which can be identified with the holomorphic functions.
 For a harmonic function $u$ in $\mathbb{R}^n$, $n \geq 3$, it is no longer true that $\log|\nabla u|$ is necessarily  subharmonic.
 However some logarithmic convexity properties for harmonic functions  still hold. One example is the classical three spheres theorem, which claims that  for solutions $u$ to a reasonable uniformly elliptic equation $Lu=0$ in $\mathbb{R}^n$ (one can think that $L=\Delta$)
the following inequality holds

\begin{equation}\sup_{B}|u|\le C(\sup_{\frac{1}{2}B}|u|)^\gamma(\sup_{2B}|u|)^{1-\gamma},\end{equation}
 where $B=\{x \in \mathbb{R}^n: |x|\leq 1 \}$, constants $C>0,\gamma \in (0,1)$ depend only on the elliptic operator $L$ and do not depend on $u$.

 The three spheres theorem holds for linear uniformly elliptic PDE of higher order under some smoothness assumptions on the coefficients (\cite{S})
 as well as for some  non-linear elliptic equations (\cite{D}).


\subsection{Three spheres theorem for wild sets.}

 Throughout this paper $\Omega$ will be a bounded domain in $\R^n$ and $u$ will denote a solution of an elliptic equation in the divergence form $\dv(A\nabla u)=0$ in $\Omega$ with Lipschitz coefficients.  We will show that in the three spheres theorem one can replace $\sup_{\frac{1}{2}B}|u|$ by the supremum
 over any set $E$ with positive volume.

Let $E$ and $\K$ be subsets of $\Omega$ such that the distances from $E$ and $\K$ to $\partial\Omega$ are positive. We assume that
$E$ has positive n-dimensional Lebesgue measure. 
 We aim to prove the following estimate
\begin{equation}\label{eq:main}\sup_{\K}|u|\le C(\sup_E|u|)^\gamma(\sup_\Omega|u|)^{1-\gamma},\end{equation}
 where $C>0$ and $\gamma \in(0,1)$ are independent of $u$, but depend on $\Omega$, $A$, the measure of $E$, and the distances from $\K$ and $E$  to the boundary of $\Omega$.  

 If $\sup_\Omega|u|=1$ and $\sup_E|u|=\varepsilon$, then \eqref{eq:main} can be written as 
\begin{equation}\label{eq:eps}
\sup_{\K}|u|\le C\varepsilon^\gamma.
\end{equation} 
 This inequality explains why the result is called propagation of smallness. Typically, we start with some set, where we know that the solution is small, and then we make a conclusion that it is also small on a bigger set.

The fact that the set $E$ is allowed to be arbitrary wild, while the constants depend only on its measure, seems to be useful for applications, see \cite{AEWZ}. Further we will specify the dependence of  constants on the  measure in the form of the Remez type inequality.

\subsection{Preceding results.}
 The result that we prove is expected. We would like to mention the preceding work in this direction. 
In the case of analytic coefficients the estimate \eqref{eq:main} was proved by Nadirashvili \cite{N1}, see also \cite{V1}. The case of $C^\infty$-smooth coefficients remained open till now, but there were several attempts to prove it. Estimates, weaker than \eqref{eq:eps}, were obtained by Nadirashvili \cite{N2} and Vessella \cite{V2}. See also \cite{MV}, where the case of solutions of elliptic equations with singular lower order coefficients is treated.
 In the preceding results the exponent $\varepsilon^\gamma$ in the right-hand side of \eqref{eq:eps} was replaced by $\exp(-c|\log\varepsilon|^p)$ for some $p=p(n)<1$. We push $p$ to $1$ in this paper.  

\subsection{Remez type inequality.}
In this note we prove \eqref{eq:main} in the  setting of smooth coefficients, using the new results on the behavior of the doubling index of solutions to elliptic equations presented in \cite{LM1, L2, L3}.
 On the way of proving \eqref{eq:main} we obtain an interesting inequality for solutions of elliptic equations, which reminds the classical Remez inequality for polynomials, the role of the degree is now played by the doubling index.

 Let $Q$ be a unit cube. Assume $u$ is a solution to  $\dv(A\nabla u)=0$ and the doubling index $N= \log  \frac{\sup_{2Q} |u|}{\sup_Q |u|}$.
 Then 
\begin{equation} \label{Remez1}
\sup_{Q}|u|\le C \sup_E|u|\left(C\frac{|Q|}{|E|}\right)^{CN}
\end{equation}
where $C$ depends on $A$ only, $E$ is any subset of $Q$ of a positive measure.

Note that if $u$ is a harmonic polynomial in $\mathbb{R}^n$,
 then one can replace $N$ by the degree of $u$. The doubling index for harmonic polynomials
 can be estimated from above by the degree of the polynomial.

Garofalo and Lin \cite{GL} proved almost monotonicity of doubling index for solutions of second order elliptic PDEs and applied this result
to prove unique continuation properties. In particular, they showed that  both $|u|^2$  and $|\nabla u|^2$ are Muckenhoupt weights with parameters that depend on the maximal doubling index. This implies \eqref{Remez1} with some implicit power $const(N)$ in place of $CN$ 
 and with $L^2$ norm in place of $\sup$  norm. 


\subsection{Propagation of smallness from sets with big Hausdorff dimension.}

The assumption that $E$ has positive $n$-dimensional Lebesgue measure can be relaxed. It is enough to assume that the dimension of $E$ is larger than $n-1$, see \cite{M} for the details in the analytic case.
We  fix the  Hausdorff content of $E$ of some order  $n-1+\delta$ with $\delta>0$ and obtain inequality \eqref{eq:main}. 
The main Lemma \ref{pr:2}  gives an upper estimate  for the Hausdorff content of the set where the solution is small.

\subsection{Propagation of smallness for gradients.}
 In  Section \ref{sec:grad} we prove a result for the gradients of solutions of elliptic PDEs, which is new even for ordinary harmonic functions in $\mathbb{R}^n,$ $n\geq 3$. Propagation of smallness for the gradients of solutions is better than for the solutions themselves. More precisely, 
 the inequality remains the same
\begin{equation} \label{eq:grad1}
\sup_{\K}|\nabla u|\le C(\sup_E|\nabla u|)^\gamma(\sup_{\Omega}|\nabla u|)^{1-\gamma},
\end{equation}
but now  the set $E$ is allowed to be smaller.
Namely, we show that there is a constant $c=c(n)\in (0,1)$ such that \eqref{eq:grad1} is valid for 
sets $E$ with Hausdorff dimension $$\dim_\h(E) > n-1-c.$$ We give the precise statement in Section \ref{sec:grad}.

\textbf{Precaution.} We warn the reader that the paper is not self-contained: sometimes we use recent results, which are proved in other papers.
 Namely,  we use the technique of counting doubling indices developed in \cite{LM1,L2,L3} and in Section \ref{sec:grad} we rely on estimates for sublevel sets of the gradients of solutions obtained in \cite{CNV}.

\subsection{Open questions.}
We  propagate smallness (for gradients) from  sets of Hausdorff dimension  bigger than $n-1-c$.   It would be interesting to obtain quantitative estimates for propagation  of smallness from sets of Hausdorff dimension greater than $n-2$. 
There are qualitative stratification  results for critical sets \cite{CNV} that suggest that $n-2$ is the correct threshold, but at the moment  there are no  quantitative  estimates  for the interval $(n-1-c, n-2)$.

\textbf{Question 1.} Is it true that the inequality \eqref{eq:grad1} holds for sets $E$ with $\dim_\h(E) > n-2$ and the constants can be chosen to depend only on the operator $A$, domain $\Omega$, the distances from $E$ and $\K$ to the boundary of $\Omega$ and the Hausdorff content of $E$ of order $n-2+\delta$ for any $\delta>0$?

Such estimates would be related to a conjecture by Fang-Hua Lin \cite{Lin} on the size of the critical sets of  solutions. For the sake of simplicity we formulate Lin's conjecture
 for ordinary harmonic functions, we also slightly modify the definition of the frequency. 

\textbf{Conjecture 1} (Fang-Hua Lin). Let $u$ be a non-zero harmonic function in the unit ball $B_1\subset \mathbb{R}^n$, $n\geq 3$.   
 Consider $$N= \log \frac{\sup_{B_{1}}|\nabla u|}{\sup_{B_{1/2}}|\nabla u|}$$
   Is it true that the Hausdorff measure  $$\h^{n-2}( \{ \nabla u=0\}\cap B_{1/2} )\leq C_n N^2 $$
 for some $C_n$ depending only on the dimension?

An interesting topic in the propagation of smallness which we don't touch in this paper is the dependence of the constants in \eqref{eq:main} and \eqref{eq:grad1} on the distance from the set $E$ to the boundary of $\Omega$. 

\textbf{Question 2.} Consider the inequality \eqref{eq:grad1} with a set $E$ of $\dim_\h(E) =n-1$ and fixed $\K$, $A$ and $\Omega$. How do the constants $C$ and $\gamma$ depend on the  distance from $E$  to the boundary of $\Omega$?

This question is connected to the quantitative version of the Cauchy uniqueness problem, see \cite{Lin} for related results when $E$ is a relatively open subset of the boundary. The situation changes when we consider wild sets on the boundary of positive surface measure. The  following  question is quite famous, it dates back to at least L. Bers. The two-dimensional case is not difficult due to connections with complex analysis. 
  The fact that the question is open in higher dimensions shows that  we still  don't understand well the Cauchy uniqueness problem even for ordinary harmonic functions
in the dimension three or higher (which is quite embarrassing for the well-developed theory of elliptic PDEs nowadays).

\textbf{Conjecture 2.} Assume that $u$ is a harmonic function in the unit ball $B_1\subset \mathbb{R}^3$ and $u$ is $C^\infty$-smooth in the closed ball $\overline{B_1}$.
 Let  $S \subset \partial B_1 $ be any closed set with positive area.
Is it true that $\nabla u = 0$ on $S$ implies $\nabla u \equiv 0$?

 Usually this question is asked in the form of the Cauchy uniqueness problem, where the condition $\nabla u = 0$ is replaced by
the condition that the Cauchy data $(u,\frac{\partial u}{\partial n})$ are zero on $S$. 
If one takes any Lebesgue point of $S$, then harmonicity of $u$ and $C^\infty$-smoothness automatically implies that all the derivatives of $u$ of any order
are zero at this point. Since the area of Lebesgue points of $S$ is the same as of $S$, one can also assume (in the question above) that all the derivatives
of $u$ vanish at the boundary subset of positive area and the question is whether the harmonic function $u$ should be identically zero.

 For the class  $C^{1+\varepsilon}( \overline{B_1} )$ there is a striking counterexample \cite{BW}, which however is not $C^\infty$-smooth up to the boundary. The attempts to construct $C^2$-smooth counterexamples were not successful.


\subsection{Estimates for  Laplace eigenfunctions}
 Let $(M,g)$ be a $C^\infty$ smooth closed Riemannian manifold  and let $\Delta$ denote the Laplace operator on $M$.
Consider the sequence of Laplace eigenfunctions $\varphi_\lambda$ on $M$ with $\Delta \varphi_\lambda + \lambda \varphi_\lambda = 0. $

We would like to make a remark that the Remez type inequality \eqref{Remez1} for harmonic functions  implies
the following bound for Laplace eigenfunctions, which was conjectured in \cite{DFR}.  For any subset $E$ of $M$ with positive volume the following holds:
\begin{equation} \label{Remez2}
 \sup_E|\varphi_\lambda| \geq \frac{1}{C}\sup_{M}|\varphi_\lambda |\left(\frac{|E|}{C|M|}\right)^{C\sqrt \lambda},
\end{equation}
where $C=C(M,g)>1$ does not depend on $E$ and $\lambda$. 
 Note that $\sqrt \lambda$ corresponds to the degree of the polynomial in Remez inequality.

Looking at the following example of spherical harmonics $u(x,y,z)= \Re (x+iy)^n$ one can see that $L^2$ norm of restriction of $u$ on the unit sphere  is concentrated near equator very fast and $|u|$ is exponentially small on most of the unit sphere.
 This example shows that a sequence Laplace eigenfunctions can be $e^{-c\sqrt \lambda}$ small on  a fixed open subset of the manifold.
 
 The proof of implication \eqref{Remez1} $\implies $ \eqref{Remez2} is a standard trick, we give a sketch of the proof of the implication, which is not difficult.

 \textit{ The function $u(x,t)=\varphi_\lambda(x)e^{\sqrt \lambda t}$ is a harmonic function on the product manifold $M\times \mathbb{R}$. 
  The  doubling index $N$  of $\varphi$ in any geodesic ball is smaller than $C_1\sqrt \lambda$ (\cite{DF}). Then the doubling index for $u$ in any geodesic ball of radius smaller than the diameter of $M$ is also smaller than $C_2\sqrt \lambda$. One can apply \eqref{Remez1} to $u$ with N=$C_2\sqrt \lambda$ and get  the bound \eqref{Remez2} for $\varphi$.}

 It seems that for negatively curved Riemannian manifolds one can prove better versions of \eqref{Remez2}. We don't feel the curvature in our methods.

We would like to mention an outstanding recent result from the works by Bourgain \& Dyatlov \cite{BD} and Dyatlov \& Jin \cite{DJ}.
\begin{theorem}[\cite{BD},\cite{DJ}]
 Under assumption that $(M,g)$ is a closed Riemannian surface with constant negative curvature the following inequality holds for Laplace eigenfunctions on $M$.
 Given an open subset $E$ of $M$ there exists $c=c(E,M,g) >0$  such that 
$$ \int_E \varphi_\lambda^2 \geq c  \int_M \varphi_\lambda^2.$$
\end{theorem}
 The constant $c$ does not depend on the eigenvalue $\lambda$.
Note that the situation on closed surfaces of constant negative curvature is different from the case of the sphere.

 A beautiful result \cite{BR} by Bourgain and Rudnick states  that on a two dimensional torus $T^2 = \mathbb{R}^2 / \mathbb{Z}^2$ 
equipped with the standard metric the toral Laplace eigenfunctions $\varphi_\lambda$ satisfy $L^2$ lower and upper restriction bounds on curves.
Namely, if $S$ is a smooth curve on $T^2$ with non-zero curvature and $\lambda > const(S)$, then $$ c \| \varphi_\lambda \|_{L^2(S)} \leq \|\varphi_\lambda \|_{L^2(T^2)}  \leq C \|\varphi_\lambda \|_{L^2(S)}.$$
In particular that implies that on a given  smooth curve, which is not geodesic, only a finite number of Laplace eigenfunctions can vanish. 
 
 A very interesting question that we don't touch here   is how $L^2$ mass of Laplace eigenfunctions $\varphi_\lambda$  are asymptotically distributed on the manifold as $\lambda \to \infty$. In particular,  for negatively curved surfaces the quantum unique ergodicity conjecture states that asymptotically the $L^2$ mass of eigenfunctions is distrubed uniformly.  We refer to \cite{QUE1},\cite{Li},\cite{QUE2},\cite{QUE3},\cite{QUE4},\cite{QUE5},\cite{DJ} for the results on ergodic properties of eigenfunctions.

\section{Preliminaries} \label{sec:mr}
\subsection{Hausdorff content}
Remind that the Hausdorff content of a set $E\subset\R^n$ is
\[C_\h^d(E)=\inf
\big\{
\sum_j r_j^d: E\subset\cup_j B(x_j, r_j)\big\},\]
and the Hausdorff dimension of $E$ is 
defined as
$$\dim_\h(E)=\inf\{d: C_\h^d(E)=0\}.$$
Clearly the Hausdorff content is sub-additive 
$$C_\h^d(E_1\cup E_2)\le C_\h^d(E_1)+C_\h^d(E_2).$$
It also satisfies the natural scaling identity, if $\phi_t$ is a homothety of $\R^n$ with coefficient $t$ then $$C_\h^d(\phi_t(E))=t^dC^d_\h(E).$$  The advantage of the Hausdorff content  over the corresponding Hausdorff measure is that the former is always finite on bounded sets, it is bounded from above by $\diam(E)^d$. The Hausdorff content of order $n$ is equivalent to the Lebesgue measure.

\subsection{Three spheres theorem for wild sets.}
We always assume that $u$ is a solution of an elliptic equation in divergence form in a bounded domain $\Omega \subset \mathbb{R}^n$,
\begin{equation}\label{eq:ell}\dv(A\nabla u)=0,\end{equation}
where $A(x)=[a_{ij}(x)]_{1\le i,j\le n}$ is a symmetric uniformly elliptic matrix with Lipschitz entries,
\begin{equation}\label{eq:Lambda}
\Lambda_1^{-1}\|\zeta\|^2\le\langle A\zeta, \zeta\rangle\le \Lambda_1\|\zeta\|^2,\quad |a_{ij}(x)-a_{ij}(y)|\le\Lambda_2|x-y|.
\end{equation}

Let $m,\delta, \rho$ be positive numbers. Suppose a set $E\subset \Omega$ satisfies $$C_\h^{n-1+\delta}(E)> m, \quad \dist(E,\partial \Omega) > \rho.$$ 
Let $\K$  be a  subset of $\Omega$ with $\dist(\K,\partial \Omega) > \rho$. Our main result is the following.

\begin{theorem}\label{th:m} There exist
 $C, \gamma>0$, depending on $m,\delta, \rho, A,\Omega$ only  such that 
$$\sup_{\K}|u|\le C(\sup_E|u|)^\gamma(\sup_\Omega|u|)^{1-\gamma}$$
for any solution $u$ of $\dv(A\nabla u)=0$ in $\Omega$.
\end{theorem}


\subsection{Doubling index }\label{sec:aux}
 We formulate several well-known lemmas connected to the three spheres theorem (or monotonicity property of the frequency function of a solution). We refer to \cite{HL} for an introduction to the frequency function, which is almost a synonym for the doubling index (the
 term "`frequency"' will not be used in this paper). 



 Let $B$ be a ball in $\R^n$. Define the doubling index of a non-zero function $u$ (defined in $2B$) by 
$$ N(u, B)=  \log\frac{\sup_{2B} |u|}{\sup_B |u|}.$$
It is a non-trivial fact  that the doubling index of  solutions to an elliptic second order PDE in divergence form is almost monotonic in the following sense:
 \begin{equation}\label{eq:monot}N(tB) \leq N(B) (1+c) +C\end{equation}
for any positive $t \leq 1/2$. Here as usual $tB$ denotes a ball of radius $t$ times the radius of $B$ with the same center as $B$, the constants $c,C>0$ depend on $A$, but are independent of $u$.  Almost monotonicity of the doubling index implies the three spheres theorem. The  three spheres theorem that implies the almost monotonicity property of the doubling index  was proved in  the work \cite{L} of Landis. Garofalo and Lin (\cite{GL}) proved a sharper version  of the monotonicity property. In particular, the results of \cite{GL}  imply that if the elliptic operator is a small perturabtion of the Laplace operator, then $c$ in \eqref{eq:monot} can be chosen to be small ($C$ is still big, but it is less important).
  We refer the reader to \cite{Man} and \cite{L2} for further discussion.

 For a cube $Q$ in $\R^n$ let $s(Q)$ denote its side length and let $tQ$ be the cube with the same center as $Q$ and such that $s(tQ)=ts(Q)$.
 Suppose that $(20n)Q\subset \Omega$. We define the doubling index of a function $u$ in the cube $Q$  by
\begin{equation}\label{eq:double}
N(u,Q)=\sup_{x\in Q, r\le s(Q)}\log\frac{\sup_{B(x, 10nr)}|u|}{\sup_{B(x,r)}|u|}.
\end{equation} 
 
  This is a kind of maximal version of the doubling index, which is convenient in the sense that if a cube $q$ is a subset of $Q$, then 
 $N(u,q) \leq N(u,Q)$. The definition implies the following estimate. Let $q$ be a subcube of $Q$ and $K =\frac{s(Q)}{s(q)} \geq 2.$
Then 
\begin{equation} \label{eq:3sph}
\sup_q |u| \geq c K^{-C N} \sup_Q |u|,
\end{equation}
 where  $N=N(u,Q)$ and $c$ and $C$ depend on $n$ only. 




\section{Auxiliary lemmas}
\subsection{Estimates  of the zero set} The doubling index is useful for estimates of the zero set of solutions of elliptic equations. We will need the following known result.

 \begin{lemma}[\cite{HS}]\label{lem:up} Let $u$ be a solution to  $\dv(A\nabla u)=0$ in $\Omega\supset 20nQ$.
 For any $N>0$ there exists $C_N$, which is independent of $u$, but depends on $A$,$N$ and $\Omega$, such that if $ N(u,Q)\leq N$, then
\begin{equation} \label{lem:3sph}
\h^{n-1}(\{u=0\}\cap Q) \leq C_N s(Q)^{n-1}.
\end{equation}
\end{lemma}
  We will use only finiteness of $C_N$ and will apply it for $N$ smaller than some numerical constant.
 \begin{remark} One can ask what is the optimal upper bound. The harmonic counterpart of the Yau conjecture suggests that there is a  linear estimate:
 $$ \h^{n-1}(\{u=0\}\cap Q) \mathop{\leq}\limits^{?} C_{\Omega, A} N.$$
The conjecture is open, but known in the case of analytic coefficients due to results by Donnelly and Fefferman \cite{DF}. 
In the setting of smooth coefficients an exponential bound ($CN^{CN}$) was proved in \cite{HS}, a recent result in \cite{L2}
provides a polynomial upper bound $C N^{\alpha}$, $\alpha>1$ depends on the dimension. 
 \end{remark}

The measure of the zero set can be also estimated from below. We assume that $u$ and $Q$ are as in Lemma \ref{lem:up}.
 
\begin{lemma} \label{lem:low}
 Let $q$ be a subcube of $Q$ and suppose that $u$ has a zero in $q$.
 Then 
\begin{equation}\label{eq:weak}
\h^{n-1}(\{u=0\}\cap 2q) \mathop{\geq} c_N  s^{n-1}(q),
\end{equation}
where $c_N$ depends on $A,\Omega$ and $N=N(u,Q)$.
\end{lemma}
\begin{remark}
The following  much stronger version of this estimate is proved in \cite{L3},
 $$ \h^{n-1}(\{u=0\}\cap 2q) \mathop{\geq} c  s^{n-1}(q)$$
 where $c$ depends on $A,\Omega$ only. We will use only the weak inequality  (\ref{eq:weak}) above, which is not difficult (see for example \cite{LM1}). 
\end{remark}

  
\subsection{Estimate for  sub-level sets} 
 The following lemma gives an estimate for the size of the set where a solution to an elliptic PDE is small in terms of the doubling index. We note that the lemma below is qualitative, but not quantitative (in a sense that there is no control of constants in terms of $N$). The lemma  will be further refined to a quantitative version (Lemma \ref{pr:2}).

\begin{lemma}\label{l:base} Let $\delta\in (0,1], N >0$. Assume that $u$ satisfies $\dv(A\nabla u)=0$ in $(20n)Q$,  $\sup_{Q}|u|=1$ and $N(u,Q) \leq N$. Let 
$$E_a=\{x\in \frac{1}{2}Q: |u(x)|<e^{-a}\}.$$ 
Then 
\[C_\h^{n-1+\delta}(E_a)\le M e^{-\beta a}s(Q)^{n-1+\delta},\]
for some $\beta=\beta(N, \delta, A,\Omega)$ and $ M =M (N,\delta, A,\Omega)$.
\end{lemma}
\begin{proof}
 By $c, C, \kappa, c_1,C_1 \dots$ we will denote positive constants that depend on $\delta, A,$ and $\Omega$ only, while constants $c_N$, $C_N$  additionally depend on $N$.
 
Clearly, it is enough to prove the statement for $ a >> N$ and $a>>1$. For small $a$ the inequality holds if we  choose $M$ large enough to satisfy the inequality.
Without loss of generality we may  assume that $N \geq 2$.

 Let $ K = [e^{\kappa a /N}] $ where $\kappa>0$ is a sufficiently small constant to be specified later.   
Partition $\frac{1}{2}Q$ into $K^n$ equal 
subcubes $q_i$. We will assume that $K>4$, then $4q_i \subset Q$.
 We will estimate the number of cubes $q_i$ that intersect $E_a$.

Let $q_i$ be a cube with $q_i \cap E_a \neq \emptyset$. So $\inf_{q_i} |u| \leq e^{-a} $.

  Assume first that $u$ does not change sign in $2q_i$.
Then by the Harnack inequality
 $$\sup_{q_i}|u| \leq c_1 \inf_{q_i}|u| \leq c_1 e^{-a}.$$
On the other hand by  \eqref{eq:3sph} we have
$$ \sup_{q_i}|u| \geq c_2 K^{- C_1 N}  \geq \frac{c_2}{2} e^{-C_1 \kappa a} .$$
Now, we specify  $\kappa = \frac{1}{2C_1}$. Then the two inequalities above cannot coexist for large $a$. 
 
Hence if $q_i$  intersects $E_a$, then $u$  changes sign in $2q_i$. 
   Denote by $S$ the set of cubes $q_i$ such that $u$  changes sign in $2q_i$.
  Note that 
\begin{equation}\label{eq:L}
 C_\h^{n-1+\delta}(E_a) \leq C_2 |S|s(Q)^{n-1+\delta}K^{-n+1-\delta}. 
\end{equation}
   Now, we will estimate $|S|$ using the bounds for the size of the zero set of $u$. Note that $u$ has a zero in each $2q_i$ for $q_i\in S$. Recall that $4q_i \subset Q$ and each point in $Q$ may be covered only by a finite number of $4q_i$, depending only on the dimension.
 By Lemma \ref{lem:low}
 $$ \h^{n-1}( \{u=0\} \cap Q) \geq c_3 \sum_{S} \h^{n-1}( \{u=0\} \cap 4q_i) \geq c_4 c_N |S| s(Q)^{n-1}K^{-n+1}.$$
 On the other hand, by Lemma \ref{lem:up}
$$ \h^{n-1}( \{u=0\} \cap Q)  \leq C_Ns(Q)^{n-1}.$$
 We therefore have 
$$|S| \leq \frac{C_N}{c_4 c_N} K^{n-1}.$$
 Thus by \eqref{eq:L}
 $$ C_\h^{n-1+\delta}(E_a) \leq C_3 \frac{C_N}{ c_N} K^{-\delta}s(Q)^{n-1+\delta} \leq C_4 \frac{C_N}{ c_N} e^{- \kappa \delta a/N}s(Q)^{n-1+\delta},$$
which is the required estimate with $\beta=\kappa\delta/N$ and $M=C_4C_Nc_N^{-1}$.
\end{proof}

\begin{remark}
 In \cite{L3, L2} it was shown that one can choose $c_N$ independent of $N$ and $C_N= C N^{\alpha}$, where $\alpha$ depends only on the dimension. Hence for $N \geq 1$,
  \begin{equation*}
C_\h^{n-1+\delta}(\{|u|< e^{-a}\}\cap \frac{1}{2}Q) \leq  C N^{\alpha} e^{- c \delta a/N }s(Q)^{n-1+\delta}.
 \end{equation*}
The optimal estimates for 
 $c_N$ and $C_N$ will appear to be not necessary for the purposes of this paper.
In Lemma \ref{pr:2} we will prove a better bound for $C_\h^{n-1+\delta}(E_a)$
 without using  the uniform lower bound for $c_N$ or polynomial bound for $C_N$.
 
\end{remark}


\subsection{Main tool}
The following lemma  will be severely exploited in the proof of main results. See Section ``Number of cubes with big doubling index'' in \cite{L3} for the proof of the lemma formulated below. We note that the definition of the doubling index in \cite{L3} is slightly different (but the proof remains the same).

\begin{lemmaA*} 
 Let $u$ be a solution to  $\dv(A\nabla u)=0$ in $\Omega$. There exist positive constants $ s_0, N_0, B_0$ that depend on $A,\Omega$ only such that if $Q$ is a cube with $s(Q)<s_0$, $(20n)Q\subset \Omega$, and $Q$ is divided into $B^n$ equal subcubes with $B>B_0$, then the number of subcubes $q$ with $N(u,q)\ge \max (\frac{1}{2} N(u,Q), N_0)$ is less than $B^{n-1-c}$, where $c$ depends on the dimension $n$ only.
\end{lemmaA*}
\begin{remark}  If we are interested in sets of positive Lebesgue measure only, it would be enough to apply this result with a weaker 
bound on the number of subcubes with large doubling index, namely $B^{n-c}$, see the combinatorial lemma in \cite{LM1}, which is simpler. 
\end{remark}

\section{Proof of the Main result}\label{sec:proof}
\subsection{Reformulations of Theorem \ref{th:m} } 
Clearly Theorem \ref{th:m} is a local result. We formulate an equivalent local version. 
\begin{prop}\label{pr:1}  Let $\Omega$ be a bounded domain in $\R^n$, $A$ satisfy \eqref{eq:Lambda} and $\delta$ and $m$ be positive. There exist
 $C, \gamma>0$, depending on $ A,\Omega, m$ and $\delta$  such that the following holds. Suppose that $u$ is a solution to $\dv(A\nabla u)=0$ in  $\Omega\supset (10n^2)Q$ and let $E\subset \frac{1}{20n} Q$ satisfy $C_\h^{n-1+\delta}(E)\ge ms(Q)^{n-1+\delta}$, then
\[\sup_{Q}|u|\le C(\sup_E|u|)^\gamma(\sup_{(10n^2)Q}|u|)^{1-\gamma}.\]
\end{prop}
The constants $20n$ and $10n^2$  are for technical purposes only. One can replace them by the constant $2$ and the lemma above will remain true. 

One can use the standard argument to deduce Theorem \ref{th:m} from Proposition \ref{pr:1}. We give only a sketch without details. First, find a suitable cube $Q$ with $20n^2 Q \subset \subset \Omega $ and $ C_\h^{n-1+\delta} (\frac{1}{20n} Q \cap E) >0$. Second, apply Proposition \ref{pr:1}. It shows that we can propagate smallness from $E$ onto the cube $Q$. Third, with the help of the three spheres theorem  the standard Harnack chain argument allows to  propagate  smallness from 
$Q$ onto the whole $\K \subset\subset \Omega$.

It remains to prove Proposition \ref{pr:1}, which will follow from the next lemma. All the main ideas of the paper are used  in the proof of the  lemma, reduction of the proposition to the lemma will be given below and is not difficult.
\begin{lemma} \label{pr:2} Suppose that $\dv(A\nabla u)=0$ in $(20n)Q$ and $\sup_{Q}|u|=1$. Let  $N =N(u,Q) \geq 1$. Set as above
$$E_a=\{x\in \frac{1}{2}Q: |u(x)|<e^{-a}\}.$$ 
Then \[C_\h^{n-1+\delta}(E_a)<Ce^{-\beta a/N}s(Q)^{n-1+\delta},\]
for some $C,\beta>0$ that depend on $A,\delta$ only.

\end{lemma}
\begin{remark}
This lemma with $\delta=1$ can be written as a version of Remez inequality (see \cite{R}) for solutions of $\dv(A\nabla u)=0$ and the role of the degree of a polynomial is played by the doubling index:
\[\sup_{Q}|u|\le C \sup_E|u|\left(C\frac{|Q|}{|E|}\right)^{CN}\]
where  $C$ depends on $A$ only, $E$ is a subset of $Q$ of a positive measure and $N=N(u,Q)$ is defined by \eqref{eq:double}.  Note that one can also replace the maximal version of the doubling index $N(u,Q)$ by $\log \frac{\sup_{2Q} |u|}{\sup_Q |u|}$ and the statement will remain true. The standard reduction, which we omit,  uses the monotonicity property of the doubling index.
\end{remark}

\subsection{Lemmas \ref{l:base} and \ref{pr:2} imply Proposition \ref{pr:1}.}
 Consider two cases.

 First case: $N=N(u,\frac{1}{10n}Q) \leq 1$. 
Here Lemma \ref{l:base} is applicable for $\frac{1}{10n}Q$ and since $C_\h^{n-1+\delta}(E)>m$ we have
 $$\sup_E |u| \geq c_m \sup_{\frac{1}{10n}Q}|u|.$$
 And by the three spheres (squares) theorem we know
$$ \sup_{Q}|u|\le C(\sup_{\frac{1}{10n}Q}|u|)^\gamma(\sup_{(10n^2)Q}|u|)^{1-\gamma}. $$
 
 Second case: $N=N(u, \frac{1}{10n}Q) \geq 1$.
 Assume that $C_\h^{n-1+\delta}(E)=ms(Q)^{n-1+\delta}>0$, $|u|<\varepsilon$ on $E$ and $\sup_{\frac{1}{10n}Q}|u|=1$. 
We apply Lemma \ref{pr:2} in the cube $\frac{1}{10n}Q$  with  $a=|\log \varepsilon|$. Then $E\subset E_a$ and the lemma implies that $$m<C\varepsilon^{\beta/N}$$ and therefore $$N\ge \gamma|\log\varepsilon|,$$ where $\gamma=\gamma(C, m, \beta)$. 

 It is time to use the definition of  the doubling index, see Section \ref{sec:aux}.
 There exists a ball $B_r(x)$, $x \in \frac{1}{10n}Q $, $r \leq \frac{1}{10n} s(Q)$ such that 
$$ \log \frac{\sup\limits_{B_{10nr}(x)}|u|}{\sup\limits_{B_r(x)}|u|} \geq N - 1/100 .$$

 Note that $B_{10nr}(x)\subset B_{\sqrt n s(Q)}(x)$ and $B_{\sqrt n s(Q)}(x)$ also contains $Q$. Then the monotonicity of the doubling index \eqref{eq:monot} and the assumption $N\ge 1$ implies
\[ \log \frac{\sup\limits_{(10n^2)Q}|u|}{\sup\limits_{ Q}|u|} \geq\log \frac{\sup\limits_{B_{10n\sqrt n s(Q)}(x)}|u|}{\sup\limits_{ B_{ \sqrt n s(Q)}(x)}|u|} \geq c_1 \log \frac{\sup\limits_{B_{10nr}(x)}|u|}{\sup\limits_{B_r(x)}|u|} \geq c_2N\ge c_2\gamma|\log\epsilon|\] 
Thus Proposition \ref{pr:1} follows. It remains to prove Lemma  \ref{pr:2}.

\subsection{Proof of Lemma \ref{pr:2}}
Now, the ellipticity and Lipschitz constants (see \eqref{eq:Lambda}) $\Lambda_1 \geq 1$  and $\Lambda_2>0$ are  fixed parameters and  $Q_0$ is the unit square in $\R^n$. Numbers $N > 1$ and $a>0$ are variables. 
Let $$m(u,a)=C_\h^{n-1+\delta}\{x\in Q_0: |u(x)|<e^{-a}\sup_{Q_0}|u|\},$$ and
 \[M(N,a)=\sup_{*}m(u,a),\] where the supremum is taken over all elliptic operators $\dv(A \nabla \cdot )$ and functions $u$ satisfying the following conditions in $20n Q_0$:
\begin{itemize}
 \item[(i)]  $A(x)=[a_{ij}(x)]_{1\le i,j\le n}$ is a symmetric uniformly elliptic matrix with Lipschitz entries satisfying \eqref{eq:Lambda},
 \item[(ii)] $u$ is a solution to $\dv(A\nabla u)=0$ in $20n Q_0$,
\item[(iii)] $N(u,Q_0)\le N$. 
\end{itemize}
Our aim is to show that
\begin{equation}\label{eq:pro}
M(N,a)\le Ce^{-\beta a / N}.
\end{equation}
 The constant $\beta>0$ will be chosen later and will not depend on $N$.

 We can always assume that $$a/N \gg 1$$ by making the constant $C$ sufficiently large.
 By Lemma \ref{l:base} we can also assume that $N$ is sufficiently large, in particular $N/2 \geq N_0$, where $N_0$ is the constant 
 from Lemma A.

The proof  contains several steps.
 First, with the help of Lemma A we  prove a recursive inequality for $M(N,a)$. Then we show how this inequality implies the exponential bound  \eqref{eq:pro} by a double induction argument on $a,N$. 
\subsection*{Recursive inequality.} We show that
\begin{equation} \label{eq:rec}
 M(N,a) \leq B^{1-\delta}  M( N/2, a - C_1 N \log B )+ B^{-\delta -c} M(N, a -C_1 N \log B).
\end{equation}
The constant $C_1$ will be specified later; we choose  $B=B_0+1$ and $c$ from Lemma A.

 Fix  a solution $u$ to the elliptic equation $\dv(A \nabla u)=0$  with $N(u,Q_0)\le N$. Divide $Q_0$ into $B^n$ subcubes $q$.
 Lemma A claims that we can partition cubes $q$ into two  groups: a group of good cubes with $N(u,q)  \leq N/2$ and a
 group of bad cubes with $N/2 \leq N(u,q)  \leq N$ such that the number of all bad cubes is smaller than $B^{n-1-c}$ (and the number of all good cubes is smaller than the total number of cubes $B^{n}$). We have
\begin{equation*} \label{eq:m1}
m(u,a)\le  \sum\limits_{q}C_\h^{n-1+\delta}(\{x\in q: |u(x)|<e^{-a}\sup_{Q_0}|u|\}).
\end{equation*}
 By  (\ref{eq:3sph}) we see that
\begin{equation}\label{eq:sup}
\sup_{q}|u|\ge c_1 B^{-C_1 N}\sup_{Q_0}|u|.
\end{equation} 
 Since $N$ is sufficiently large, we can forget about $c_1$ above by increasing $C_1$. We continue to estimate $m(u,a)$:
\begin{equation} \label{eq:m2}
m(u,a)\le  \sum\limits_{q}C_\h^{n-1+\delta}(\{x\in q: |u(x)|<e^{-a} B^{C_1 N}\sup_{q}|u|\})
\end{equation}
\begin{equation*} \label{eq:m3}
=  \sum_{\text{good\ } q}+\sum_{\text{bad\ } q} C_\h^{n-1+\delta}(\{x\in q: |u(x)|<e^{-\tilde a} \sup_{q}|u|\})
\end{equation*}
where 
\begin{equation*}
\tilde a = a -  C_1 N \log B.
\end{equation*}
 Now, we estimate each sum individually. If $q$ is a good cube, then
\begin{equation*}
C_\h^{n-1+\delta}(\{x\in q: |u(x)|<e^{-\tilde a} \sup_{q}|u|\}) \leq B^{-(n-1+\delta)} M(N/2, \tilde a)
\end{equation*}
 Above we used the scaling property of $C_\h^{n-1+\delta}$ and the fact that the restriction of $u$ to a cube $q$ corresponds to a solution of another elliptic PDE in the unit cube, the new equation can be written in the divergence form with some coefficient matrix  which satisfies the same estimate \eqref{eq:Lambda}.

Since the total number of good cubes is smaller than $B^n$ 
\begin{equation*}
\sum_{\text{good\ } q} \leq B^{1-\delta} M(N/2, \tilde a)
\end{equation*}
 We know that the number of bad  cubes $q$ is smaller than $B^{n-1-c}$. Hence
\begin{equation*}
\sum_{\text{bad \ } q} \leq B^{n-1-c} B^{-(n-1+\delta)} M(N, \tilde a) = B^{-c-\delta} M(N, \tilde a) .
 \end{equation*}
 Adding the inequalities for bad and good cubes and taking the supremum over $u$, we obtain the recursive inequality \eqref{eq:rec} for $M(N,a)$.

\subsection*{Recursive inequality implies  exponential bound.}

We will now prove that 
\begin{equation}\label{eq:pro2}
M(N,a)\le Ce^{-\beta a / N}.
\end{equation}
 by a double induction on $N$ and $a$.
 Without loss of generality we may assume $N=2^l$, where $l$ is an integer number. 
 Suppose that we know \eqref{eq:pro2} for $N=2^{l-1}$ and all $a>0$ and now we wish to establish it 
 for $N=2^{l}$.  By Lemma  \ref{l:base} we may assume 
 $l$ is sufficiently large. So we can say that  Lemma  \ref{l:base} gives the basis for the induction.
 For a fixed $l$ we argue by induction on $a$ with step  $C_1N\log B$. Recall that $B$ is a sufficiently large number for which Lemma A holds. 
 We will assume that $a>> N$, namely $a > C_0N \log B$, where $C_0>0$ will be chosen later.
  For $a<C_0N \log B$ the inequality is true if $C$ is  large enough.

 By the induction assumption we have $$M(N,a - C_1N\log B)\le Ce^{-\beta a / N + C_1 \beta \log B}$$
and 
$$ M(N/2,a - C_1N\log B)\le Ce^{-2 \beta a / N + 2C_1 \beta\log B}.$$

Finally, we use the recursive inequality \eqref{eq:rec} and get

$$M(N,a) \leq CB^{1-\delta}e^{-2\beta a/N+2C_1\beta \log B}+CB^{-\delta-c}e^{-\beta a/N + C_1\beta\log B}. $$
 Our goal is to obtain the following inequality 
$$ B^{1-\delta}e^{-2\beta a/N+2C_1\beta \log B}+B^{-\delta-c}e^{-\beta a/N+ C_1\beta \log B} \leq e^{-\beta a/N}$$
for $a/N>C_0\log B$.
Dividing by $e^{-\beta a/N}$ we reduce it to
\[B^{1-\delta+2C_1\beta}e^{-\beta a/N}+ B^{-\delta-c+C_1\beta}\le 1.\]
Now, recall that $a/N>C_0\log B$
and the last inequality follows from
\[B^{1-\delta+2C_1\beta - C_0\beta }+ B^{-\delta-c+C_1\beta}\le 1.\]

The last inequality can be achieved with the proper choice of the parameters: $B>2,\delta, c, C_1>0$ are fixed,  we choose $\beta$ to be small enough so that the second term is less than $1-\varepsilon$ and then choose large $C_0$ to make the first term smaller than $\varepsilon$. Thus the inequality above holds for all sufficiently  large  $a/N$.
As we mentioned above, for small $a/N$  the inequality $\eqref{eq:pro2}$ is true if we choose $C$ to be large.

\begin{remark}
One can notice that the induction step is working for negative $\delta$ such that $ -c <\delta$.
However the induction basis step (Lemma \ref{l:base}) is not true for negative $\delta$. For instance, zeroes of harmonic functions in $\mathbb{R}^n$ are sets of dimension $n-1$.
But the induction basis step appears to be true for gradients of solutions, which have better unique continuation properties than the solutions itself.  
\end{remark}

\section{Propagation of smallness for the gradients of solutions} \label{sec:grad}
\subsection{Formulation of the result}
As above we assume that $u$ is a solution of an elliptic equation \eqref{eq:ell} in divergence form in a bounded domain $\Omega \subset \mathbb{R}^n$ and the coefficients satisfy \eqref{eq:Lambda}. 
\begin{theorem}\label{th:m2}  There exists a constant $c \in (0,1)$ that depends only on the dimension $n$ such that the following holds. Let $m,\delta, \rho$ be positive numbers and suppose  sets $E,\K \subset \Omega$ satisfy $$C_\h^{n-1-c+\delta}(E)> m, \quad \dist(E,\partial \Omega) > \rho, \quad
 \dist(\K,\partial \Omega) > \rho.$$ Then there exist
 $C, \gamma>0$, depending on $m,\delta, \rho, \Lambda_1, \Lambda_2,\Omega$ only (and independent of $u$) such that 
$$\sup_{\K}|\nabla u|\le C(\sup_E|\nabla u|)^\gamma(\sup_{\Omega}|\nabla u|)^{1-\gamma}.$$
\end{theorem}

%
\subsection{Modifications of the proof.}
 We shall use the notion of doubling index for $|\nabla u|$. Let $B=B(x_0,r)$ be a ball in $\mathbb{R}^n$. Define
$$ N(\nabla u,B)=   \log\frac{\sup_{2B} |\nabla u|}{\sup_B |\nabla u|}.$$
 Assume $r \leq 1$. The doubling index is almost monotonic:
 \begin{equation} \label{eq:c}
N(\nabla u,tB) \leq N(\nabla u,B) (1+c) +C 
 \end{equation}
for $t \leq 1/2$. The constants $c,C>0$ depend on $\Lambda_1,\Lambda_2$ (the ellipticity and Lipschitz constants) and the dimension $n$.
The monotonicity of the doubling index for $|\nabla u|$ follows from the three spheres theorem for the function $|u(\cdot) - u(x_0)|$
 and standard elliptic estimates.  A similar modification appeared in \cite{CNV}, see also \cite{GL}.
 We also need a modified doubling index for a cube $Q$:

\[N(\nabla u,Q)=\sup_{x\in Q, r\le s(Q)}\log\frac{\sup_{B(x, 10nr)}|\nabla u|}{\sup_{B(x,r)}|\nabla u|}.\]

The proof of Theorem \ref{th:m2} is parallel to the proof of Theorem \ref{th:m}. 
We need to establish  analogs of    Lemma A (induction step), Lemma \ref{l:base} (basis of induction),  and Lemma \ref{pr:2} (estimate of the Hausdorff content), where $|u|$ should be replaced by $|\nabla u|$. We formulate such statements below.

\begin{lemmaB*} 
 There exist positive constants $ s_0, N_0, B_0$ 
 that depend on $\Lambda_1,\Lambda_2$ and the dimension $n$ only such that if $Q$ is a cube with  side $s(Q)<s_0$ and $Q$ is divided into $B^n$ equal subcubes with $B>B_0$, then the number of subcubes $q$ with $N(\nabla u,q)\ge \max (\frac{1}{2} N(\nabla u,Q), N_0)$ is less than $B^{n-1-c}$, where $c \in (0,1)$ depends on the dimension $n$ only.
\end{lemmaB*}
\begin{lemma} \label{l:base2} Let $Q_0$ be the unit cube in $\mathbb{R}^n$. Suppose that $\dv(A\nabla u)=0$ in $(20n)Q_0$, $\sup_{Q_0}|\nabla u|=1$, and $N(\nabla u,Q_0)\le N_0$, then for 
$$E_a=\{x\in Q_0: |\nabla u(x)|<e^{-a}\}$$ 
we have \[C_\h^{n-2+\delta}(E_a)<Ce^{-\beta a}\] 
for some $\beta,C$ depending on $N_0$, $\Lambda_1, \Lambda_2$, $\delta$.
\end{lemma}
\begin{lemma} \label{pr:3} Let $Q_0$ be the unit cube in $\mathbb{R}^n$. Suppose that $\dv(A\nabla u)=0$ in $(20n)Q_0$ and $\sup_{Q_0}|\nabla u|=1$. Let a number $N =N(\nabla u,Q_0) \geq 1$. Set 
$$E_a=\{x\in \frac{1}{2}Q_0: |\nabla u(x)|<e^{-a}\}.$$
There exists $c\in (0,1)$ that depends only on the dimension $n$ such that   
 \[C_\h^{n-1-c+\delta}(E_a)<Ce^{-\beta a/N},\]
for some $C,\beta>0$ that depend on $\Lambda_1, \Lambda_2, \delta, n$ only.
\end{lemma}
 Only the proof of Lemma \ref{l:base2} requires modifications, the other changes are minor.

\subsection{Outline of changes.}
The reduction of Theorem \ref{th:m2} to Lemma \ref{pr:3} is not difficult and remains the same as in Section \ref{sec:proof}. To prove Lemma \ref{pr:3} one has to replace the used Lemma \ref{l:base}  by its analog for $|\nabla u|$ (Lemma \ref{l:base2}), which we  prove below. The proof is based on new results from \cite{CNV}.

 The proof of Lemma B repeats the proof of Lemma A (\cite{L3}). There are two main ingredients in the proof: simplex lemma and hyperplane lemma from \cite{L2}. We don't formulate those lemmas here, see \cite{L2}.
There are no changes in the proof of hyperplane lemma, except that one has to subtract a constant from the function.

 To prove the simplex lemma we need a sharper version of the monotonicity property  of the doubling index as it was in the proof of the original simplex lemma. Namely, one has to make $c$ in inequality \eqref{eq:c}  a sufficiently small constant, depending only on the dimension. One has to make  a linear change of coordinates such that $A(0)$ turns into  the identity matrix and
 $\Lambda_1$ is close to $1$ in a small neighborhood of the origin. After that one obtains a sharper version of the three spheres theorem for $u-u(0)$ as it is done in \cite{L2}. Then one should use standard elliptic estimates to provide smallness of $c$ in \eqref{eq:c}.

  To prove Lemma \ref{pr:3} one has to use the same induction argument as in Lemma \ref{pr:2}.
The induction step remains the same, but one has to work with the doubling index for $|\nabla u|$ in place of $|u|$ and use Lemma B in place of Lemma A. Concerning the basis of induction, which is  Lemma \ref{l:base2}, a different argument is needed, and we will use a result  from \cite{CNV}, which estimates the size of the neighborhood of the effective critical set. 
That would give us an analog of Lemma \ref{l:base} for $|\nabla u|$, but now the dimension of the set $E$ will be allowed to be smaller than $n-1$, but bigger than $n-2$.
Unfortunately, the induction step works only for dimensions bigger $n-1-c$ only and that is the main obstacle for improvement towards $n-2$.


\subsection{Proof of Lemma \ref{l:base2}}
The lemma is a corollary from Theorem 1.17 (estimate of the effective critical set) from \cite{CNV}. We warn the reader that we formulate it below in  our own notation and don't bring the proof of Theorem 1.17.
  \subsubsection*{Reformulation of Theorem 1.17 from \cite{CNV}.}
 Let $u$ be as in Lemma \ref{l:base2}.
 For any $\delta>0$ there exist positive constants  $C$ and $c$ depending on $n,\Lambda_1,\Lambda_2, \delta,N_0$ such that the following holds for all integer $K$. 
Partition the unit cube $Q_0$ into $K^n$ sub-cubes $q$ with side length $1/K$. We call q bad if
$$ \inf_q |\nabla u | < c \sup_{2q}|\nabla u|.$$
Then the number of bad cubes $q$ 
is not greater than $C K^{n-2+\delta} $.

 Now, we are ready to finish the proof of Lemma \ref{l:base2}.
We divide the unit cube $Q_0$ into $K^n$ sub-cubes $q$ with side length $1/K$, the integer  $K$ will be chosen later.

The monotonicity of the doubling index implies
$$ \sup_q |\nabla u| \geq c_1 K^{- C_1N_0- C_1} \sup_{Q_0} |\nabla u| =c_1 K^{- C_1N_0- C_1}. $$
If $q$ is not bad, then 
$$ \inf_q |\nabla u| \geq c_2 K^{- C_1N_0- C_1}.$$

Given $a>0$ we  want to estimate the Hausdorff content $C_\h^{n-2+2\delta}$ of 
$$E_a=\{x\in Q_0: |\nabla u(x)|<e^{-a}\}.$$
We may assume $a>1$.
 Now, we specify the choice of $K$. The $K$ is smallest integer number greater than $2$ such that  $$e^{-a}>c_2 K^{- C_1N_0- C_1}.$$
 So $\log K  $ is comparable to $a$. And the set $E_a$ is contained in the union of bad cubes of size $1/K$. The  number of bad cubes is not greater than 
$C K^{n-2+\delta} $. 
 We therefore have  
 \[C_\h^{n-2+2\delta}(E_a)\le C_2 K^{\delta} \le C_3 e^{- c_3 a}.\]
 Replacing $2\delta$ by $\delta$ we finish the proof.

\subsection*{Acknowledgment}
This work was partly done when the second author held a one year visiting position at the Department of Mathematics at Purdue University and it is a pleasure to thank the department for its hospitality.
  

\end{document}